\documentclass[10pt, oneside]{amsart}

\usepackage[margin=1.4in]{geometry}

\usepackage{amssymb}
\usepackage[all]{xy}
\usepackage{tikz-cd}
\usepackage[T1]{fontenc}
\usepackage{lmodern}
\usepackage{xfrac}
\usepackage{mathtools}

\usepackage[
textwidth=3cm,
textsize=small,
colorinlistoftodos]
{todonotes}

\tikzset{ampersand replacement=\&}

\newcommand{\mack}{\mathcal{M}\!\textit{ack}} %category of Mackey functors
\newcommand{\Sp}{\mathcal{S}\!\textit{p}} %category of spectra
\newcommand{\Orb}{\mathrm{Orb}}  %orbit
\newcommand{\Orbx}{\mathrm{Orb}^{\times}} %orbit^x
\newcommand{\dga}{\mathrm{DGA}}
\newcommand{\cdga}{\mathrm{CDGA}}
\newcommand{\Comm}{\mathrm{Comm}}
\newcommand{\gr}{\mathrm{gr}}   %graded
\newcommand{\CDGA}{\mathrm{CDGA}}
\newcommand{\alg}{\mathrm{alg}}
\newcommand{\Enaivealg}{E^1_\infty\text{-}\alg}
\newcommand{\EGalg}{E^G_\infty\text{-}\alg}

\newcommand{\QQ}{\mathbb{Q}} %the rationals

\newcommand{\KU}{\textit{KU}} %Periodic topological K-theory
\newcommand{\ku}{\textit{ku}} %connective topological K-theory
\newcommand{\RU}{\textit{RU}} % Representation ring
\newcommand{\M}{\underline{M}} % A general Mackey functor
\newcommand{\rep}{\underline{\RU}} %complex representation ring Mackey functor without Q

\DeclareMathOperator{\id}{id} %identity map

\newcommand{\abs}[1]{\lvert{#1}\rvert}

\usepackage[capitalise]{cleveref}

\newcommand{\newrefformat}[2]{}

\newtheorem{thm}{Theorem}[section] % numbered theorem
\newtheorem{cor}{Corollary}[section]
\makeatletter\let\c@cor\c@thm\makeatother
\newtheorem{lemma}{Lemma}[section]
\makeatletter\let\c@lemma\c@thm\makeatother
\newtheorem{prop}{Proposition}[section]
\makeatletter\let\c@prop\c@thm\makeatother

\makeatletter\let\c@claim\c@thm\makeatother
\newtheorem{thrm}{Theorem}[section]
\makeatletter\let\c@thrm\c@thm\makeatother

\newtheorem*{unnumberedtheoremA}{Theorem A}  % unnumbered theorem
\newtheorem*{unnumberedtheoremB}{Theorem B}

\theoremstyle{definition}

\newtheorem{defn}{Definition}[section]
\makeatletter\let\c@defn\c@thm\makeatother
\newtheorem{const}{Construction}[section]
\makeatletter\let\c@const\c@thm\makeatother

\makeatletter\let\c@notn\c@thm\makeatother

\makeatletter\let\c@outline\c@thm\makeatother

\makeatletter\let\c@recoll\c@thm\makeatother

\makeatletter\let\c@strategy\c@thm\makeatother

\theoremstyle{remark}
\newtheorem{rem}{Remark}[section]
\makeatletter\let\c@rem\c@thm\makeatother
\newtheorem{ex}{Example}[section]
\makeatletter\let\c@ex\c@thm\makeatother

\makeatletter\let\c@observation\c@thm\makeatother

\newcommand{\A}{\mathcal{A}}
\newcommand{\bQ}{\mathbb{Q}}

\newcommand{\lra}{\longrightarrow}
\newcommand{\Ch}{\mathsf{Ch}}
\newcommand{\SH}{\mathsf{SH}}

\makeatletter
\let\c@equation\c@thm
\numberwithin{equation}{section}
\makeatother

\title[Genuine-commutative structure on rational $K$-theory]{Genuine-commutative structure on rational equivariant $K$-theory for finite abelian groups}
\author[Bohmann]{Anna Marie Bohmann}
\address[Bohmann]{Vanderbilt University}
\author[Hazel]{Christy Hazel}
\address[Hazel]{University of California Los Angeles}
\author[Ishak]{Jocelyne Ishak}
\address[Ishak]{Vanderbilt University}
\author[K\k{e}dziorek]{Magdalena K\k{e}dziorek}
\address[K\k{e}dziorek]{Radboud University Nijmegen}

\author[May]{Clover May}
\address[May]{University of California Los Angeles}
\begin{document}

\begin{abstract}
In this paper, we build on the work from \cite{BHIKM} to show that periodic rational $G$-equivariant topological $K$-theory has a unique genuine-commutative ring structure for $G$ a finite abelian group.
This means that every genuine-commutative ring spectrum whose homotopy groups are those of $KU_{\mathbb{Q},G}$ is weakly equivalent, as a genuine-commutative ring spectrum, to $KU_{\mathbb{Q},G}$.  In contrast, the connective rational equivariant $K$-theory spectrum does not have this type of uniqueness of genuine-commutative ring structure.

\end{abstract}

\maketitle

\section{Introduction}\label{sect:intro}

Periodic equivariant topological $K$-theory $KU_{G}$ is one of the foundational cohomology theories of algebraic topology.  It was first defined by Segal \cite{Segal}. Equivariant $K$-theory for a group $G$ links stable homotopy theory with the complex representation theory of $G$ and all of its subgroups. For example, induction and restriction of representations are visible in the structure of the zeroth stable homotopy groups of $KU_{G}$. Equivariant complex $K$-theory arises from consideration of vector bundles on manifolds and provides deep connections between representation theory and geometric topology.

Equivariance has come to the forefront of stable homotopy theory in recent years, as researchers have realized the powerful possibilities of equivariant structures.  These were made manifest in Hill, Hopkins, and Ravenel's ground-breaking solution to the Kervaire invariant one problem \cite{HHR}.  The formulation of this problem does not involve group actions in any way, but the solution in \cite{HHR} relies fundamentally on the techniques of equivariant homotopy theory.  An essential tool in this work is an algebraic structure of multiplicative transfer maps that is present in the homotopy groups of a $G$-equivariant cohomology theory with a strongly commutative cup product, when $G$ is finite. This structure is part of a hierarchy of levels of commutativity in the equivariant world that describe the extent to which homotopy commutativity is compatible with the group action.  This hierarchy has been known for some time \cite{ConstenobleWaner,McClureTate} but had not been extensively used.  Systematic study of these levels of commutativity began with work of Blumberg and Hill on $N_\infty$-operads \cite{BlumbergHill}, and it has since been a very active area of research \cite{BlumbergHill2, BlumbergHill3, Rubin, GutierrezWhite, BonventrePereira, Bohme1, Bohme2, BBRNinftyassociahedra}. While this body of work has gone a long way towards illuminating equivariant commutativity, these structures are still widely regarded as subtle and complicated.

In this paper, we analyze the rationalization of equivariant complex $K$-theory $\KU_{\mathbb{Q},G}$ from the perspective of two levels of commutativity, where $G$ is finite abelian.   Equivariant $K$-theory enjoys the maximal level of commutativity, sometimes called \emph{genuine commutativity}, as shown in \cite{Joachim}.  Forgetting structure, we may also view it as having the minimal level of commutativity, \emph{naive commutativity}.  The results here and of our previous work in \cite{BHIKM} show that rational periodic equivariant $K$-theory is homotopically unique when considered at both levels.  In contrast, the non-periodic version of rational complex $K$-theory is only unique when considered at the minimum level of commutativity.

The difference between these levels of commutativity boils down to the presence of \emph{norm maps} on the homotopy groups of a spectrum.  After rationalization, these norm maps have a particularly elegant and approachable form.   Our uniqueness result---the first of its kind in equivariant stable homotopy theory---illustrates the power of algebraic models for rational homotopy theory.  It also serves as an approachable and illuminating window into the structures that distinguish the varying levels of equivariant commutativity.

Our proof uses a recent result of Wimmer \cite{Wimmer} that provides an algebraic model for the category of rational genuine-commutative ring $G$-spectra when $G$ is a finite group.  In particular, Wimmer shows that the $\infty$-category of rational genuine-commutative ring $G$-spectra, which have all norms, is equivalent to the category $[\Orb_G, \CDGA(\bQ)]$ of functors from the orbit category of $G$ to the category of rational commutative differential graded algebras.  The norm maps that characterize genuine-commutative ring spectra
show up in the algebraic model as maps of CDGAs $A_\bullet(G/H)\to A_\bullet(G/K)$ arising from maps $G/H\to G/K$ in the orbit category that are not isomorphisms.
In this way, Wimmer's algebraic model for genuine-commutative ring spectra extends the algebraic model for naive-commutative ring spectra developed by Barnes, Greenlees, and K\k{e}dziorek in \cite{BarnesGreenleesKedziorek}---their work shows that rational naive-commutative ring spectra are modeled by functors from $\Orb^\times_G$ to $\CDGA(\QQ)$, where $\Orb^\times_G$ denotes the subcategory of isomorphisms in $\Orb_G$.

In previous work \cite{BHIKM}, we calculated the image of rational equivariant $K$-theory in the algebraic model for naive-commutative ring spectra, when $G$ is an abelian group.  In fact, the calculation of \cite{BHIKM} follows from a uniqueness result in that category akin to the main result of the present paper.  The homotopy groups of topological complex $K$-theory form a commutative Green functor and were computed previously by Segal \cite{Segal}. When $G$ is abelian, these homotopy groups have trivial Weyl group actions and we used this to show that there is a unique naive-commutative ring spectrum with this commutative Green functor of homotopy groups. In this paper, we further show there is a unique genuine-commutative ring spectrum with the same underlying homotopy groups. 
This is proved as Theorem \ref{genuineuniquenessKU}.
\begin{unnumberedtheoremA}
Let $G$ be a finite abelian group. If $X$ is a genuine-commutative rational ring $G$-spectrum whose underlying homotopy Green functor is isomorphic to that of $\KU_{\QQ,G}$, then $X$ is weakly equivalent to $\KU_{\QQ,G}$ as genuine-commutative ring spectra.
\end{unnumberedtheoremA}

The assumption in Theorem A that $X$ be a genuine-commutative rational ring $G$-spectrum is essential, as we show in \cref{e:no_extension_to_genuine}. Working with the algebraic models for naive- and genuine-commutative ring $G$-spectra, we show that there exists a naive-commutative ring $G$-spectrum $Y$ that is equivalent to $\KU_{\QQ,G}$ as a naive-commutative ring $G$-spectrum, but does not enjoy a genuine-commutative ring structure. That is, there exists no genuine-commutative ring $G$-spectrum whose underlying naive-commutative ring $G$-spectrum is $Y$.

In \cite{BHIKM}, we also show a similar uniqueness result for the naive-commutative structure on connective $K$-theory $\ku_{\QQ,G}$.  However, in contrast to the case for periodic $K$-theory, we do not get this uniqueness at the genuine-commutative level for $\ku_{\QQ,G}$. Instead, we find that not all weak equivalences in the category of naive-commutative ring $G$-spectra can be extended to equivalences of genuine-commutative ring $G$-spectra.  This result is proved as \cref{thm:non_uniqueness_ku}.
\begin{unnumberedtheoremB}
Let $G$ be a finite abelian group. There exists a rational genuine-commutative ring $G$-spectrum $X$ whose underlying Green functor of homotopy groups is isomorphic to that of $\ku_{\QQ,G}$ but which is not weakly equivalent to $\ku_{\QQ,G}$.  That is, $X$ is weakly equivalent to $\ku_{\QQ,G}$ in the category of rational naive-commutative ring $G$-spectra but not in the category of rational genuine-commutative ring $G$-spectra.
\end{unnumberedtheoremB}
This type of behavior is typical for increasing algebraic structure: for example, the wedge of Eilenberg--MacLane spectra $\bigvee_{n} \Sigma^{2n}H\QQ$ is rationally equivalent to $\KU_\QQ$ but this does not extend to an equivalence of ring spectra.  The example of $\ku_{\QQ,G}$ is an explicit illustration of the additional rigidity of  weak equivalences that preserve genuine-commutative ring structure, in contrast to those that preserve only naive commutativity.

\subsection*{Notation and Conventions}
Throughout the paper we assume that $G$ is a finite group.  We use the notation $\A(G)$ for the algebraic model of rational $G$-spectra, $\A(E_\infty^1(G))$ for the algebraic model of rational naive-commutative ring $G$-spectra, and $\A(E_\infty^G(G))$  for the algebraic model of rational genuine-commutative ring $G$-spectra; see Definition \ref{def:AlgMod}, Definition \ref{def:naiveAlgMod}, and Definition \ref{def:genuineAlgMod} respectively. If $X$ is a rational naive-commutative ring $G$-spectrum then $\theta(X)$ denotes its derived image in $\A(E_\infty^1(G))$. If $X$ is a rational genuine-commutative ring $G$-spectrum then we denote by $\Theta(X)$ its derived image in $\A(E_\infty^G(G))$ and by $\theta(X)$  its derived image in $\A(E_\infty^1(G))$, after forgetting part of the structure.

As we exclusively work in the rationalized context, to avoid notational clutter we leave the rationalization of a spectrum implicit in our notation.  Hence $\KU_{\QQ,G}$ will typically be denoted by $\KU_G$, and likewise $\ku_{\QQ,G}$ will be denoted by $\ku_G$.
However, we maintain the rationalization in the notation for categories; for example $\SH_\QQ$ denotes the rational stable homotopy category. The only exception is in the notation for the algebraic models, which do not have non-rational counterparts.

Finally, we use $\simeq$ to denote a zig-zag of Quillen equivalences between model categories or an equivalence between $\infty$-categories.

\subsection*{Acknowledgments}
We thank Hausdorff Research Institute for Mathematics in Bonn for their hospitality in hosting the Women in Topology III workshop, where this research began. Funding for the workshop was also provided in part by Foundation Compositio Mathematica and the National Science Foundation of the United States. NSF support was via the grants NSF-DMS 1901795 and NSF-HRD 1500481:\ AWM ADVANCE. The first author was partially supported by NSF Grant DMS-1710534. The fourth author was supported by  NWO Veni Grant 639.031.757.

We also thank Mike Hill for many useful conversations and Christian Wimmer for sharing a draft of his work. The first author thanks Spencer Dowdall for acting as a notational sounding-board.

%%%%%%%%%%%%%%%%%%%%%%%%
\section{Preliminaries}\label{sect:reviewalgmodels}

In this section we recall the framework of algebraic models for rational $G$-spectra. We begin with the classical, non-equivariant story.

The process of rationalization drastically simplifies stable homotopy theory. This philosophy dates back at least to Serre.  Serre's computations of stable homotopy groups of spheres \cite{Serre1951} imply there is an equivalence between the rational stable homotopy category $\SH_\QQ$ and graded $\bQ$-vector spaces $\gr(\bQ\text{-mod})$ given by taking homotopy groups
\[ \pi_*(-)\colon
 \SH_\QQ \lra \gr(\bQ\text{-mod}).\]
This result can be lifted to an equivalence of homotopy theories, either at the level of $\infty$-categories or model categories. For example, Shipley \cite{ShipleyHZ} constructed a zig-zag of symmetric monoidal Quillen equivalences between rational spectra and rational chain complexes
\[\Sp_\bQ \simeq \Ch_\bQ. \]
Since these Quillen equivalences are symmetric monoidal, they induce an equivalence between ring spectra and differential graded algebras. Moreover, if $R$ is a rational ring spectrum, we get an induced equivalence between $R$-modules in spectra and corresponding modules in rational chain complexes. A bit more work shows that $E_\infty$-algebras in the two categories are also equivalent. They are modeled by algebras over the commutative operad  on both sides and work of Richter and Shipley \cite{RichterShipley} provides a zig-zag of Quillen equivalences between them.

Rational equivariant stable homotopy theory can similarly be encoded in algebraic models.  For equivariant spectra, homotopy groups have a richer structure given by homotopy Mackey functors.
Roughly speaking, a $G$-Mackey functor $M$ is a collection of abelian groups indexed by subgroups $H$ of $G$, together with transfer, restriction, and conjugation maps between them satisfying certain conditions. Taking homotopy groups produces an equivalence between the rational $G$-equivariant stable homotopy category and the category of rational graded $G$-Mackey functors
\[\{\pi_*^H(-)\}_{H\leq G}\colon G\SH_\QQ \lra \gr\mack(G)_\bQ.
\]
Mackey functors are purely algebraic, but can be fairly complex---see, for example, \cite{Webb_guide} for an introduction to the subject. The complexity of Mackey functor structure comes from the fact that the restrictions and transfers arising from subgroups are interrelated.  Over the rational numbers, the category of Mackey functors simplifies. Greenlees and May \cite[Appendix A]{GM_Tate} showed that rational $G$-Mackey functors split into the product over conjugacy classes of subgroups in $G$ of $\bQ[W_GH]$-modules,
\[\mack(G)_\bQ \cong \prod_{(H),\, H\leq G} \bQ[W_GH]\text{-mod},
\]
where $W_GH$ is the Weyl group of the subgroup $H$ in $G$, i.e. $W_GH=N_GH/H$.
What is more, this splitting is compatible with the corresponding splitting at the level of rational $G$-spectra using the idempotents of the rational Burnside ring for $G$. For a modern account see \cite{BarnesKedziorek}.

Combining the above, we get an equivalence of categories
\[G\SH_\QQ \simeq  \prod_{(H),\, H\leq G} \gr(\bQ[W_GH]\text{-mod}).\]
This result was later lifted to a zig-zag of Quillen equivalences between model categories by Schwede--Shipley \cite{ss03stabmodcat} and Barnes \cite{barnessplitting}, and to a zig-zag of symmetric monoidal Quillen equivalences by K\k{e}dziorek \cite{KedziorekExceptional}.
\begin{thm}\label{thm:GAlgMod}
There exists a zig-zag of symmetric monoidal Quillen equivalences
\[G\Sp_\bQ \simeq  \prod_{(H),\,H\leq G} \Ch(\bQ[W_GH]),
\]
where the symmetric monoidal product on the right-hand side is given by
$\otimes_\QQ$ on each component of the product.
\end{thm}

\begin{defn}\label{def:AlgMod}
The category  $\prod_{(H),\,H\leq G} \Ch(\bQ[W_GH])$ with the objectwise projective model structure is an \emph{algebraic model for rational $G$-spectra} and we denote it by $\A(G)$.
If $X$ is a rational $G$-spectrum, then we denote by $\theta(X)$ its derived image in $\A(G)$.
\end{defn}

As in the non-equivariant case, these symmetric monoidal Quillen equivalences induce a zig-zag of Quillen equivalences between equivariant ring spectra and algebras in $\A(G)$.  Additionally, if $R$ is an equivariant rational ring spectrum, we get an induced equivalence between $R$-modules in spectra and corresponding modules in the algebraic model $\A(G)$.

These parallels between the equivariant and non-equivariant pictures in the associative case belie the fact that the case of commutative equivariant spectra is significantly more complicated than the non-equivariant one. To begin, in the equivariant world there are distinct, non-equivalent levels of commutativity, which all forget down to $E_\infty$-structure in $\Sp$. These different levels of commutativity are modeled by Blumberg and Hill's $N_\infty$-operads \cite{BlumbergHill}. There are two special cases of $N_\infty$-operads which are important for this paper: the \emph{genuine-commutative} operad, denoted by $E_\infty^G$, and the \emph{naive-commutative} operad, denoted by $E_\infty^1$. Algebras over these operads in $G$-spectra are called \emph{genuine-commutative ring spectra} and \emph{naive-commutative ring spectra},
respectively. In particular, genuine-commutative ring spectra are characterized by having all the norm maps between their homotopy groups.  This implies that the collection of homotopy groups of a genuine-commutative ring spectrum forms a graded Tambara functor \cite{AngeltveitBohmann}.

On the other hand, there is only one level of commutativity in the algebraic model $\A(G)$: a differential graded algebra is either graded-commutative or it is not. It turns out that commutative algebras in the algebraic model $\A(G)$ classify the naive-commutative rational ring $G$-spectra.

\begin{thm}[{\cite{BarnesGreenleesKedziorek}}]\label{thm:naiveAlgMod}
There exists a zig-zag of Quillen equivalences from naive-com\-mu\-ta\-tive rational ring $G$-spectra to the commutative algebras in the algebraic model $\A(G)$,
\[\Enaivealg(G\Sp_\bQ) \simeq \Comm\text{-}\alg(\A(G))= \prod_{(H),\, H\leq G} \CDGA(\bQ[W_GH]).
\]
\end{thm}

\begin{rem}\label{rem:naive_orb_model}
Let $\Orb_G$ denote the orbit category of $G$, whose objects are cosets $G/H$ for all $H\leq G$ and whose morphisms are $G$-equivariant maps. Denote by $\Orb^\times_G$ the wide subcategory of $\Orb_G$ consisting of only isomorphisms. Then the category of functors from $\Orb^\times_G$ to rational commutative differential graded algebras is equivalent to the algebraic category given in Theorem \ref{thm:naiveAlgMod}, that is
\[
[\Orb^\times_G, \CDGA(\bQ)] \cong  \prod_{(H),\,H\leq G} \CDGA(\bQ[W_GH]).
\]
Notice that we recover the Weyl group actions by making the usual identification of equivariant automorphisms of $G/H$ and the Weyl group $W_GH$, and then looking at the images in $\CDGA(\bQ)$. We take this presentation as our algebraic model for naive-commutative rational ring $G$-spectra.
\end{rem}

\begin{defn}\label{def:naiveAlgMod}
The category $[\Orb^\times_G, \CDGA(\bQ)]$ with the objectwise projective model structure is an \emph{algebraic model for naive-commutative rational ring $G$-spectra} and we denote it by $\A(E_\infty^1(G))$.
If $X$ is a naive-commutative rational ring $G$-spectrum then we denote by $\theta(X)$ its derived image in $\A(E_\infty^1(G))$.
\end{defn}

Recent work of Wimmer \cite{Wimmer} incorporates the additional structure of norm maps into the above algebraic model, providing an algebraic model for genuine-commutative rational ring $G$-spectra. In this case, the model is implemented by an equivalence of $\infty$-categories.
\begin{thm}[{\cite{Wimmer}}]\label{thm:genuineAlgMod}
There is an equivalence of $\infty$-categories
\[ \EGalg(G\Sp_\bQ) \simeq  [\Orb_G, \CDGA(\bQ)].
\]
\end{thm}

\begin{defn}\label{def:genuineAlgMod}
The  infinity category $[\Orb_G, \CDGA(\bQ)]$ is an \emph{algebraic model for genuine-commutative rational ring $G$-spectra} and is denoted $\A(E_\infty^G(G))$.
If $X$ is a genuine-commutative rational ring $G$-spectrum then we denote by $\Theta(X)$ its derived image in $\A(E_\infty^G(G))$.
\end{defn}

Comparing \cref{thm:genuineAlgMod} with \cref{rem:naive_orb_model}, we see that the additional norm map structure on genuine-commutative ring spectra is present in the algebraic model as maps between the CDGAs corresponding to non-conjugate subgroups of $G$. We can understand such maps as follows. Suppose $L$ and $K$ are non-conjugate subgroups and there is a map $f\colon G/L\to G/K$ in $\Orb_G$ with $eL\mapsto gK$. The equivariance of this map implies $g^{-1}Lg\leq K$, and so we can factor $f$ as an isomorphism composed with a quotient map, that is, as the composition of
\begin{align*}G/L\to G/(g^{-1}Lg),&\quad  eL\mapsto g(g^{-1}Lg)\quad \text{ and }\\
G/(g^{-1}Lg)\to G/K, &\quad e(g^{-1}Lg)\mapsto eK.
\end{align*}
Thus, as compared to an object in $\A(E_\infty^1(G))$, the additional structure of an object in $\A(E_\infty^G(G))$ is given by the image of the quotient maps. We refer to this additional structure as the \emph{shadows of the norm maps}; an individual such map is a \emph{norm shadow}.

\begin{defn}\label{def:shadowsofnorms}
Suppose $H\lneq K \leq G$. We refer to the image of the quotient map $p\colon G/H\to G/K$, $p(eH)=eK$ under a functor in $\A(E_\infty^G(G))$ as a \emph{norm shadow}.
%\emph{shadow of norm map}.
\end{defn}

\begin{rem} Unlike the actual norm maps on homotopy groups of a genuine-commutative ring $G$-spectrum, the shadows of the norm maps are maps of commutative algebras, so in particular they are maps of degree $0$ that are both additive and multiplicative. They are also appropriately equivariant, which in this context means that they commute with the images of the conjugacy maps in $\Orb_G$.
\end{rem}

\begin{rem}\label{rem:abelian_construction} Our focus in this paper is when $G$ is a finite abelian group. In this case, the action of conjugation on the collection of subgroups of $G$ is trivial, and so every map in $\Orb_G$ can be factored as an automorphism followed by a quotient map. The automorphisms are again given by elements of Weyl groups, but now the normalizer of any subgroup is the entire group $G$. Hence to define an object in $\A(E_\infty^G(G))$ when $G$ is abelian, we can specify an object of $\CDGA(\QQ[G/H])$ for each subgroup $H\leq G$, together with $G$-equivariant maps of $\CDGA$s whenever there is proper subgroup containment, subject to the requirement that these maps form a functor out of the subgroup lattice of $G$.
These maps are exactly the shadows of the norm maps. 
 We use this perspective to define an object in $\A(E_\infty^G(G))$ in Section \ref{sect:goodresolutions}.
\end{rem}

The algebraic model for naive-commutative rational ring $G$-spectra may also be derived  from the work of Wimmer at the level of $\infty$-categories. Hence combining the algebraic models for naive- and genuine-commutative ring $G$-spectra with the algebraic model for rational $G$-spectra produces a commuting diagram of $\infty$-categories.

\begin{lemma}\label{lem:diagramOfAlgMod}
The following diagram of $\infty$-categories, in which all vertical arrows denote forgetful functors,
\[\xymatrix{ G\Sp_\bQ \ar[r]^{\theta} & \A(G)\\
\Enaivealg(G\Sp_\bQ) \ar[r]^{\theta} \ar[u] & \A(E_\infty^1(G))\ar[u] \\
\EGalg(G\Sp_\bQ) \ar[r]^{\Theta}  \ar[u]& \A(E_\infty^G(G)), \ar[u] }
\]
commutes up to a natural equivalence.
\end{lemma}

\begin{proof}
The above diagram can be rewritten as the composite of the following squares.
\[\xymatrix{ G\Sp_\bQ \ar[r]^-{\Phi^-}  & [\Orbx_G, \Sp_\QQ] \ar[r]^-{\widetilde{\theta}} & \A(G)\\
\Enaivealg(G\Sp_\bQ) \ar[r]^-{\Phi^-} \ar[u] & [\Orbx_G, \Comm\Sp_\QQ] \ar[r]^-{\widetilde{\theta}} \ar[u] & \A(E_\infty^1(G))\ar[u] \\
\EGalg(G\Sp_\bQ) \ar[r]^-{\Phi_{\Comm}^-} \ar[u] & [\Orb_G, \Comm\Sp_\QQ] \ar[r]^-{\widetilde{\Theta}}  \ar[u]& \A(E_\infty^G(G)) \ar[u] }
\]
Here the right horizontal maps $\widetilde{\theta}$ and $\widetilde{\Theta}$ at the $\infty$-category level are the maps on diagrams $[\Orbx_G,-]$ and $[\Orb_G,-]$ induced by the zig-zags of Quillen equivalences from \cite{ShipleyHZ} between $\Sp_\QQ$ and $\Ch_\QQ$, and from \cite{RichterShipley} between $\Comm\Sp_\QQ$ and $\CDGA(\bQ)$.  Commutativity of the square with the two maps $\widetilde{\theta}$ follows, since the diagram at the level of spectra commutes by \cite{RichterShipley}. The bottom right square commutes, since the zig-zag of Quillen equivalences given in \cite{RichterShipley} extends to the category of diagrams of commutative spectra indexed by $\Orbx_G$ and by $\Orb_G$.

The top left square commutes because the geometric fixed points functors used in \cite{Wimmer} form a symmetric monoidal functor and $\Comm\Sp \simeq E_\infty^1\Sp$.

The left outer rectangle commutes up to a natural equivalence by \cite[Proposition 4.5]{Wimmer}.
The natural equivalence constructed there lifts to an equivalence of rational commutative spectra, and thus the bottom left square commutes.
\end{proof}

We finish this section with a simple example that illustrates the difference between the algebraic models for naive- and genuine-commutative ring $G$-spectra.

\begin{ex} Let $G=C_2$ and take an object $(\bQ[x], \bQ[x])$ in $\A(E_\infty^1(G))$, with the trivial Weyl group action. This corresponds to a naive-commutative rational ring $G$-spectrum, which we call $X$. To obtain a genuine-commutative rational ring $G$-spectrum we have to specify a norm shadow $n\colon \bQ[x] \lra \bQ[x]$. For example, two possible choices are to define $n=\id_{\bQ[x]}$ or to define $n$ to be the $\bQ$-map which sends $x$ to $0$. Each of these two choices for $n$ defines a rational genuine-commutative ring $G$-spectrum.  These spectra are not weakly equivalent as genuine-commutative ring $G$-spectra, but they both restrict to the same naive-commutative ring $G$-spectrum $X$.  In other words, the genuine-commutative structure is not uniquely determined by the homotopy groups nor by the underlying naive-commutative ring $G$-spectrum.
\end{ex}

%%%%%%%%%%%%%%%

\section{The image of a spectrum in the algebraic model}
\label{sect:strategy}
In this section we delve further into the properties of the algebraic models in order to explain methods for calculating the derived image of a given spectrum.

\subsection{Formality results in the algebraic models}
Recall that the derived comparison functor in \cref{thm:GAlgMod} from $G\Sp_\bQ$ to $\A(G)$ is denoted $\theta$. This functor is determined abstractly, and for $X$ a rational $G$-spectrum, it does not provide an explicit description of $\theta(X) \in \A(G)$.  Nevertheless, we have the following formula linking homology and homotopy groups
\[  H_*(\theta_K (X)) \cong \pi_*(\Phi^K(X)),
\]
where $\Phi^K(X)$ denotes the geometric $K$-fixed points of $X$ for $K\leq G$, and for ease of notation we have denoted the value of the functor $\theta(X)$ at $K$ by $\theta_K(X)$.  This identification allows for an explicit description in the presence of formality.  Concretely, in the case where the chain complex $\theta_K (X)$ is formal for every $K\leq G$, i.e. $\theta_K (X) \simeq H_*(\theta_K (X))$ in $\Ch(\bQ[W_GK])$, these weak equivalences at each $K\leq G$ assemble to produce a description of the derived image of $X$ in the algebraic model $\A(G)$.

By construction, a similar statement is true for naive-commutative rational ring $G$-spectra. That is, let $X$ be a rational naive-commutative ring $G$-spectrum and let $\theta(X)$ denote its derived image in the algebraic model $\A(E_\infty^1(G))$. Then we have the same formula linking homology and homotopy groups: for each $K\leq G$,
\[  H_*(\theta_K (X)) \cong \pi_*(\Phi^K(X)),
\]
where $\Phi^K(X)$ again denotes the geometric $K$-fixed points of $X$.  If it happens that for every $K\leq G$ the CDGA $\theta_K (X)$ is formal, i.e. $\theta_K (X) \simeq H_*(\theta_K (X))$ in $\CDGA(\bQ[W_GK])$,  we obtain a description of the derived image of $X$ in the algebraic model  $\A(E_\infty^1(G))$.

In \cite{BHIKM}, we used this method to calculate the derived images of $H\rep_G$, $\KU_G$ and $\ku_G$ in the algebraic model $\A(E^1_\infty(G))$.  The spectrum $H\rep_G$ is the Eilenberg--MacLane spectrum for the complex representation ring Mackey functor $\rep_G$, and $\underline{\pi}_0 KU_G = \rep_G$.  We obtained the following calculations of these CDGAs and their homology. These calculations serve as the input for the computations in the genuine algebraic model.
\begin{thrm}[\cite{BHIKM}]\label{mainresultsprevpaper} For a finite abelian group $G$, there are weak equivalences in  $\A(E^1_\infty(G))$ making the following identifications for each $K\leq G$:
\begin{align*}
&\theta_K(H\rep_G)\simeq H_*(\theta_K(H\rep_G))\simeq V_K\\
&\theta_K(\KU_G)\simeq H_*(\theta_K(\KU_G)) \simeq V_K[\beta^{\pm 1}]\\
&\theta_K(\ku_G)\simeq H_*(\theta_K(\ku_G))\simeq V_K[\beta],
\end{align*}
where the $\QQ[W_GK]$-module $V_K$ is given by
\[V_K =\begin{cases} \QQ(\zeta_n)  & \text{if $K$ is cyclic of order $n$}\\
0 & \text{else}.  \end{cases}
\]
Here $\QQ(\zeta_n)$ is the field extension of $\QQ$ by a primitive $n$th root of unity and $\beta$ is the Bott element in degree two.  The $W_GK$-actions on $V_K$ and the generators $\beta$ are trivial in all cases.
\end{thrm}

When working with genuine-commutative ring spectra, by construction and by \cref{lem:diagramOfAlgMod}, we again have the formula linking homology and homotopy groups: for each $K\leq G$,
 \[  H_*(\Theta_K (X)) \cong \pi_*(\Phi^K(X)),
\]
where $\Phi^K(X)$ denotes the geometric $K$-fixed points of $X$ and $\Theta$ is the $\infty$-functor from \cref{thm:genuineAlgMod}; again $\Theta_K(X)$ denotes the value of $\Theta(X)$ at $K\leq G$.

One would like to use formality to find the image $\Theta(X)$ for a genuine-commutative ring spectrum $X$. This time, however, $\Theta(X)$ lives in the $\infty$-category $\A(E_\infty^G(G))=[\Orb_G, \CDGA(\bQ)]$. Thus the correct notion of ``formality'' is formality of the entire diagram, which cannot be separated into formality at each individual subgroup $K\leq G$ as before.   Nevertheless, if we can show the diagrammatic formality result $\Theta(X) \simeq H_*(\Theta(X))$ in $\A(E_\infty^G(G))$, then we do obtain a description of the derived image of $X$ in the algebraic model  $\A(E_\infty^G(G))$.  This is the procedure we implement in the remainder of the paper.  We construct a zig-zag of weak equivalences in $\A(E_\infty^G(G))$ from $\Theta(KU_G)$ to $H_*(\Theta(KU_G))$ in order to identify the derived image of $\KU_G$.

Our strategy will be to first prove formality of $\Theta(H\rep_G)$ via a Koszul type resolution and then lift this to resolutions giving formality of $\Theta(\KU_G)$ and $\Theta(\ku_G)$. Note this is not the most efficient way to prove formality for $\Theta(H\rep_G)$. There is a standard formality argument for any Eilenberg--MacLane spectrum using the $(-1)$-connected cover, but this approach cannot be extended to $\Theta(\KU_G)$. For clarity we include the truncation argument below along with an explanation why this simple approach cannot work for $\Theta(\KU_G)$ or  $\Theta(\ku_G)$.

\subsection{Formality for Eilenberg--MacLane spectra}
For any $A_\bullet \in \dga$ (not necessarily commutative) we can form the $(-1)$-connected cover $\tau_{\geq 0}A_\bullet$, which is zero in negative degrees, has the zero cycles $Z_0$ of $A_\bullet$ in degree zero, and is $A_i$ in degree $i$ for $i > 0$.  As in \cite{BarnesRoitzheim}, for example, this truncation can be used to show that any $\dga$ with homology concentrated in degree zero is formal. The zig-zag of quasi-isomorphisms is given by the inclusion and the quotient map to homology
\begin{center}
\begin{tikzcd}
A_\bullet \& \tau_{\geq 0} A_\bullet \arrow{r}{\simeq} \arrow{l}[swap]{\simeq} \& H_*(A_\bullet).
\end{tikzcd}
\end{center}
The second map exists because the homology is concentrated in degree zero.

In fact, given a diagram $A_\bullet \in \A(E_\infty^G(G))$  with homology concentrated in degree zero, this argument easily extends to give formality of the diagram.
\begin{prop}
Let $G$ be any finite group. Suppose $A_\bullet \in \A(E_\infty^G(G))$ has homology concentrated in degree zero.  Then $A_\bullet$ is formal.
\end{prop}

\begin{proof}
Since $A_\bullet$ has homology concentrated in degree zero, at each conjugacy class $(K)$ we have a zig-zag as above. If $G/J\to G/K$ is a map in $\Orb_G$ and $n^K_J\colon A_\bullet(J) \to A_\bullet(K)$ is the induced map on CDGAs, then we have a commutative diagram
\begin{center}
\begin{tikzcd}
A_\bullet(K) \& \tau_{\geq 0} A_\bullet(K) \arrow{r}{\simeq} \arrow{l}[swap]{\simeq} \& H_*(A_\bullet(K))\\
A_\bullet(J) \arrow{u}{n^K_J}\& \tau_{\geq 0} A_\bullet(J) \arrow{u}{\tilde{n}^K_J} \arrow{r}{\simeq} \arrow{l}[swap]{\simeq} \& H_*(A_\bullet(J)) \arrow{u}{H_*(n^K_J)}
\end{tikzcd}
\end{center}
where the middle map is simply the restriction of $n_J^K$ to the truncation.
\end{proof}

\begin{cor}\label{cortambarafunctors}
Let $G$ be a finite group and let $\M$ be a Tambara functor.  There is a unique genuine-commutative ring structure on $H\M_\QQ$, the Eilenberg--MacLane spectrum for $\M \otimes \QQ$.
\end{cor}

\begin{proof}
Since we are working with an Eilenberg--MacLane spectrum, the image in the algebraic model $\Theta(H\M_\QQ)$ has homology concentrated in degree zero.  By the previous proposition, it is formal. Thus by \cref{thm:genuineAlgMod} the genuine-commutative ring structure on $H\M_\QQ$ is uniquely determined.
\end{proof}

This corollary is a special case of Ullman's work \cite{Ullman} in which he shows that the category of $G$-Tambara functors is equivalent to the homotopy category of genuine-commutative Eilenberg--MacLane ring $G$-spectra, without rationalization.  In the rational case, this result is a straightforward consequence of the existence and the simplicity of the algebraic model.

In particular, \cref{cortambarafunctors} implies that $\Theta(H\rep_G)$ is formal. Unfortunately, the above method does not extend to $\Theta(\KU_G)$ because $\Theta(\KU_G)$ has nonzero homology in negative degrees, so no map from the connective cover can induce an isomorphism on homology. One might hope to use something like the above argument for the connective cover $\Theta(\ku_G)$ and then invert a representative for $\beta$ afterwards, but even for the connective cover, this argument fails because we are not guaranteed any map $\tau_{\geq 0} \Theta(\ku_G) \to H_*(\Theta(\ku_G))$. Thus we need a different approach to show formality.

\section{Constructing a nice resolution of $H_*(\Theta(H\rep_G))$}\label{sect:goodresolutions}
We now give an approach to formality for the Eilenberg--MacLane spectrum $\Theta(H\rep_G)$ that can be extended to give formality for $\Theta(\KU_G)$.  We construct a resolution $B_\bullet$ for the homology of $\Theta(H\rep_G)$ that has the additional property that it admits a map to any diagram $A_\bullet$ whose degree zero homology is the same as $H_0(\Theta(H\rep_G))$, and that this map induces an isomorphism on homology in degree zero. We then use this construction to prove formality of $\Theta(\KU_G)$ in \cref{sec:5}.

We first make the following observation to simplify our constructions. The diagram $\Theta(H\rep_G)$ comes with a Weyl group action at each subgroup. As observed in \cite[Remark 4.4]{BHIKM}, if the group is finite abelian, then the action on homology is trivial for all subgroups. Note this does not imply the action on the $\CDGA$ at each subgroup is trivial. However, the following lemma shows we can always choose representatives for homology classes that are fixed. Similarly, if a boundary element is fixed, we can always find an element in its preimage under the differential that is fixed under the Weyl group action.

\begin{lemma}[Averaging]\label{lem:averagingtrick} \cite[Lemma 5.2]{BHIKM}
Let $A_\bullet$ be in $\CDGA(\bQ[W_GK])$, and suppose a homology class $x\in H_*(A_\bullet)$ is fixed under the $W_GK$-action.  Then $x$ has a representative $a \in A_\bullet$ that is fixed under the $W_GK$-action. Similarly if $y\in A_\bullet$ and $d(y)$ is fixed under the $W_GK$-action, then there exists a fixed element $b$ such that $d(b)=d(y)$.
\end{lemma}

We will construct our diagram $B_\bullet$ so that $B_\bullet(K)$ has a trivial $W_GK$-action for all subgroups $K\leq G$. When we construct the map to $A_\bullet$ where $H_0(A_\bullet)\cong H_*(\Theta(H\rep_G))$, it will always be assumed that the representatives of homology classes we choose are fixed so that the map is equivariant at each level. The above lemma ensures this is possible and we use this ``Averaging Lemma'' throughout this section without further mention.

\begin{const} \label{constructH0resolution}
For an abelian group $G$, let $B_\bullet\in \A(E^G_\infty(G))$ be the diagram of rational CDGAs defined as follows.

For $K$ a cyclic subgroup of $G$, define the free graded-commutative algebra
\[B_\bullet(K)=\bigotimes_{\substack{\text{$p$ prime}\\p\mid \abs{K}}} \bigotimes_{\substack{L\leq K,\, L\neq e\\\text{$L$ cyclic $p$-group}}} \QQ[x_L]\otimes E(t_L)\]
with $\abs{x_L}=0$, $\abs{t_L}=1$, and then define differentials
\[d(t_L)=\begin{cases} \Phi_p(x_{L}) & L\cong C_p\\ x_{L'}-x_L^p & L\cong C_{p^r} \text{ with maximal proper subgroup $L'\cong C_{p^{r-1}}$}.  \end{cases}\]
We view $B_\bullet(K)$ as a $\QQ[G/K]$-CDGA with trivial action.  Observe that because the empty tensor product is the unit CDGA $\QQ$, at the trivial subgroup this construction reduces to $B_\bullet(e)=\QQ$.

For $K$ a non-cyclic subgroup of $G$, define
\[B_\bullet(K)=M_\bullet \otimes \bigotimes_{\substack{\text{$p$ prime}\\p\mid \abs{K}}}\bigotimes_{\substack{L\leq K,\, L\neq e\\\text{$L$ cyclic $p$-group}}} \QQ[x_L]\otimes E(t_L)\]
where the differentials on the generators $x_L$ and $t_L$ are as in the cyclic subgroup case and where $M_\bullet$ is the CDGA
\[M_\bullet=E(a)\]
with $\abs{a}=1$ and $d(a)=1$, which implies the homology of $M_\bullet$ is the zero ring.  Again, regard $B_\bullet(K)$ as a $\QQ[G/K]$-CDGA with trivial $G/K$ action.

We next specify the norm shadows $n_L^K\colon B_\bullet(L)\to B_\bullet(K)$ arising from subgroup inclusions $L\leq K$.  Such an inclusion implies that $\abs{L}\mid \abs{K}$ and that any cyclic $p$-subgroup of $L$ is also a cyclic $p$-subgroup of $K$.  We may thus define any $n_L^K$ via inclusions of tensor products over smaller sets into tensor products over larger sets.
That is, for a cyclic $p$-subgroup $J$ of $L$ (and thus of $K$), we send the generators $x_{J}, t_{J}\in B_\bullet(L)$ to the similarly named generators in $x_{J}, t_{J}\in B_\bullet(K)$. In the case that $L$ is not cyclic, it must also be that $K$ is not cyclic, so we just send the extra exterior generator $a\in B_\bullet(L)$ to the corresponding generator $a\in B_\bullet(K)$. This commutes with the differentials by direct inspection.

These definitions of $n_L^K$ are clearly functorial in the subgroup lattice of $G$, and thus define an $\Orb_G$-diagram in $\cdga$, as discussed in \cref{rem:abelian_construction}.  
\end{const}

\begin{lemma}\label{lem:goodreshomology} The homology of the complex $B_\bullet(K)$ defined in  \cref{constructH0resolution} is concentrated in degree 0, where it is given by
\[ H_0(B_\bullet(K))=\begin{cases}  \QQ(\zeta_{\abs{K}}) & \text{ $K$ cyclic}\\
0 & \text{ $K$ non-cyclic.}
\end{cases}\]
\end{lemma}
\begin{proof}
Consider first the case where $K$ is cyclic, so \[B_\bullet(K)=\bigotimes_{\substack{\text{$p$ prime}\\p\mid \abs{K}}} \bigotimes_{\substack{L\leq K,\, L\neq e\\\text{$L$ cyclic $p$-group}}} \QQ[x_L]\otimes E(t_L).\]
The differentials do not mix the tensor factors corresponding to distinct primes $p$, so we may write this CDGA as a tensor product of CDGAs corresponding to each prime divisor of $\abs{K}$.  As $K$ is cyclic, it has a unique maximal cyclic $p$ subgroup, which we denote $K_p$, and all cyclic $p$ subgroups of $K$ are subgroups of $K_p$.  We may thus write
\[ B_\bullet(K)=\bigotimes_{\mathclap{p\text{ prime},\, p\mid\abs{K}}}\quad B_{\bullet}(K_p).\]
By the K\"unneth theorem, the homology of $B_\bullet(K)$ splits as the tensor product
\[H_*(B_\bullet(K))\cong \bigotimes_{p \mid \abs{K}} H_*(B_\bullet(K_p)).\]
For a fixed prime $p$, let $p^k$ be the maximal power of $p$ dividing $\abs{K}$.  There is then a unique cyclic $p$ subgroup of $K$ of order $p^r$ for $1\leq r\leq k$ and we may rewrite the CDGA $B_\bullet(K_p)$ as
\[ \bigotimes_{\substack{L\leq K_p,\, L\neq e\\\text{$L$ cyclic $p$-group}}} \QQ[x_L]\otimes E(t_L)=\QQ[x_p,x_{p^2}\dots, x_{p^k}]\otimes E(t_p,t_{p^2},\dots, t_{p^k})\]
with differentials $d(t_{p^r})=\begin{cases} \Phi_p(x_p) &r=1\\ x_{p^{r-1}}-(x_{p^r})^p & r>1\end{cases}.$

A standard regular sequence argument then proves that the homology of this complex vanishes in degrees different from 0, and in degree 0 is given by
\begin{align*}H_0(B_\bullet(K_p))&\cong \QQ[x_p,\dots,x_{p^k}]/\left(\Phi_p(x_p), x_p-(x_{p^2})^p,\dots, x_{p^{k-1}}-(x_{p^k})^p\right)\\
&\cong \QQ[x_p, x_{p^k}]/\left(\Phi_p(x_p), x_p-(x_{p^k})^{p^{k-1}}\right)\\
&\cong \QQ[x_{p^k}]/\left(\Phi_{p}((x_{p^k})^{p^{k-1}})\right)\\
&\cong \QQ(\zeta_{p^k}),
\end{align*}
where the last isomorphism follows from the cyclotomic polynomial identity
\[\Phi_{p}((x)^{p^{k-1}}) = \Phi_{p^k}(x).\] We therefore have that $H_0(B_\bullet(K))\cong \bigotimes_{p\mid \abs{K}}\QQ(\zeta_{\abs{K_p}})\cong\QQ(\zeta_{\abs{K}})$.

If $K$ is not cyclic, then
\[B_\bullet(K)=M_\bullet \otimes \bigotimes_{\substack{\text{$p$ prime}\\p\mid \abs{K}}}\bigotimes_{\substack{L\leq K,\, L\neq e\\\text{$L$ cyclic $p$-group}}} \QQ[x_L]\otimes E(t_L)\]
where the differentials do not mix the $M_\bullet$ tensor factor with the remaining factors.  Since $H_*(M_\bullet)=0$, the K\"unneth theorem implies that $H_*(B_\bullet(K))=0$.
\end{proof}

\begin{lemma}\label{lem:goodres} Let $A_\bullet\in \A(E^G_\infty (G))$ be any diagram whose zeroth homology satisfies
\[ H_0(A_\bullet(K))\cong\begin{cases}  \QQ(\zeta_{\abs{K}}) & \text{ $K$ cyclic}\\
0 & \text{ $K$ non-cyclic}
\end{cases}
\]
with trivial Weyl group actions on homology.
Then there is a map in $\A(E^G_\infty(G))$
\[\psi\colon B_\bullet \to A_\bullet\]
that induces an isomorphism on $H_0$, where the diagram $B_\bullet\in \A(E^G_\infty(G))$ is the one defined in \cref{constructH0resolution}.
\end{lemma}

\begin{proof}
We construct $\psi$ by induction over the lattice of subgroups of $G$.  At the trivial subgroup $B_\bullet(e)=\QQ$, so the CDGA map $\psi_e\colon B_\bullet(e)\to A_\bullet(e)$ must be the unit map.

Let $K$ be a subgroup of $G$.  By induction, we may assume that for all proper subgroups $L\leq K$, there are Weyl-equivariant CDGA maps $\psi_L\colon B_\bullet(L)\to A_\bullet(L)$ that induce isomorphisms on $H_0$.  Moreover, we may assume that whenever $G/L'\to G/L$ is a quotient map in $\Orb_G$, the diagram
\begin{equation}\label{normfunctorialitydiagram} \xymatrix{ B_\bullet(L) \ar[r]^{\psi_L} & A_\bullet(L)\\
B_\bullet(L') \ar[r]^{\psi_{L'}} \ar[u]^{n_{L'}^L} & A_\bullet(L')\ar[u]_{n_{L'}^L}
}\end{equation}
commutes.
The proof is in four cases, depending on the structure of the subgroup $K$.\medskip

\noindent \textbf{Case 1.} \emph{$K$ is cyclic of prime order, $K\cong C_p$.}  Then $B_\bullet(K)=\QQ[x_K]\otimes E(t_K)$.  By assumption, $H_0(A_\bullet(K))\cong \QQ(\zeta_p)$ is the $p$th cyclotomic field extension.  We may thus choose a (Weyl-fixed) element $\overline{x}_K\in A_0(K)$ that represents a primitive $p$th root of unity in homology.   The homology class $[\overline{x}_K]$ must satisfy the $p$th cyclotomic polynomial $\Phi_p$, and hence $\Phi_p(\overline{x}_K)$ is the boundary of an element $\overline{t}_K\in A_1(K)$, which we may choose to be Weyl-fixed.  The assignment $\psi_K(x_K)=\overline{x}_K$ and $\psi_K(t_K)=\overline{t}_K$ thus determines a well-defined equivariant map of CDGAs
\[ \psi_K\colon \QQ[x_K]\otimes E(t_K)\to A_\bullet(K).\]
The only subgroup of $K$ is $e$, and since $B_\bullet(e)$ is the unit CDGA $\QQ$, the diagram involving the norm shadows from $e$ commutes by the uniqueness of units.\medskip

\noindent \textbf{Case 2.} \emph{$K$ is cyclic of prime power order, $K\cong C_{p^r}$, $r>1$.}   To define a map out of
\[ B_\bullet(K)=\bigotimes_{L\leq K,\ L\neq e} \QQ[x_L]\otimes E(t_L)\]
it suffices to choose Weyl-fixed images of each $x_L$ and $t_L$ that are compatible with the differentials in $B_\bullet(K)$.  When $L$ is a proper subgroup of $K$, $x_L\in B_\bullet(K)$ is in the image of the  map $n_L^K\colon B_\bullet(L)\to B_\bullet(K)$ corresponding to the quotient map $G/L\to G/K$.  Hence, in order to ensure that $\psi_K$ commutes with the norm shadows $n_L^K$, we must define $\psi_K(x_L)=n_L^K(\psi_L(x_L))$ and $\psi_K(t_L)=n_L^K(\psi_L(t_L))$, where the norm shadows on the right-hand sides are those from the diagram $A_\bullet$.  To check these elements are Weyl-fixed, observe for any $g\in G$, the diagram
\[\xymatrix{ B_\bullet(L)\ar[d]_{[g]\cdot}\ar[r]^{\psi_L} & A_\bullet(L)\ar[r]^{n_L^K} \ar[d]_{[g]\cdot}& A_\bullet(K)\ar[d]_{[g]\cdot}\\
B_\bullet(L) \ar[r]^{\psi_L}&A_\bullet(L)\ar[r]^{n_L^K} &A_\bullet(K)\\
}
\]
in which the vertical maps are the Weyl group actions by the class of $g$ in $W_GL=G/L$ or $W_GK=G/K$, must commute.  Hence the fact that $[g]$ acts trivially on $B_\bullet(L)$ implies that the elements $n_L^K(\psi_L(x_L))$ and $n_L^K(\psi_L(t_L))$ are fixed by the Weyl group of $K$ and thus  this assignment is Weyl-equivariant.  The assignments for $x_L$ and $t_L$ are compatible with the differentials on these classes in $B_\bullet(K)$ because $\psi_L$ and $n_L^K$ are maps of differential graded objects.

We thus only have to choose suitable images for $x_K$ and $t_K$. Let $J\leq K$ be the maximal proper subgroup of $K$, so $J\cong C_{p^{r-1}}$. By hypothesis $\psi_J(x_J)\in A_0(J)$ represents a primitive $p^{r-1}$th root of unity.  On homology, the norm shadow $n_J^K\colon A_\bullet(J)\to A_\bullet(K)$ induces a map
\[[n_J^K]\colon \QQ(\zeta_{p^{r-1}})\to \QQ(\zeta_{p^r}).\]
This is a map of fields, and hence is injective.  Thus the homology class of $n_J^K(\psi_J(x_J))$ represents a primitive $p^{r-1}$th root of unity in $\QQ(\zeta_{p^r}).$  It is straightforward to check (either by hand or using a splitting field argument) that any primitive $p^{r-1}$th root of unity in this field has a $p$th root: that is, there is a class $\alpha\in H_0(A_\bullet(K))$ such that $\alpha^p=[n_J^K(\psi_J(x_J))]$ and moreover, $\alpha$ is a primitive $p^r$th root of unity.  Choose a Weyl-fixed representative $\overline{x}_K\in A_0(K)$ of $\alpha$, so that $[\overline{x}_K]=\alpha$ and thus on homology we have
\[[n_J^K(\psi_J(x_J))-\overline{x}_K^p]=0.\]
This equation implies that there exists a Weyl-fixed element $\overline{t}_K\in A_1(K)$ such that $d(\overline{t}_K)=n_J^K(\psi_J(x_J))-\overline{x}_K^p$.  The assignment $\psi_K(x_K)=\overline{x}_K$ and $\psi_K(t_K)=\overline{t}_K$ is thus compatible with the differentials in $B_\bullet(K)$ by inspection.

We have defined the map $\psi_K\colon B_\bullet(K)\to A_\bullet(K)$ in such a way that for any $L\leq K$, we have the equality $n_L^K\circ \psi_L=\psi_K\circ n_L^K$, which gives a commuting square as in the diagram (\ref{normfunctorialitydiagram}).
 On homology, $\psi_K$ sends the primitive $p^r$th root of unity $[x_K]\in H_0(B_\bullet(K))$ to the primitive $p^r$th root of unity $\alpha\in H_0(A_\bullet(K))$ and is thus an isomorphism.\medskip

\noindent\textbf{Case 3.} \emph{$K$ is cyclic of non prime-power order.}  In this case,
\[B_\bullet(K)=\bigotimes_{\substack{\text{$p$ prime}\\p\mid \abs{K}}} \bigotimes_{\substack{L\leq K,\, L\neq e\\\text{$L$ cyclic $p$-group}}} \QQ[x_L]\otimes E(t_L)\]
and any cyclic $p$ subgroup $L$ in $K$ must be proper.  Hence all the generators $x_L$ and $t_L$ are in the image of norm shadows from proper subgroups.  As in Case 2, for any $p$-cyclic subgroup $L$ of $K$, we must define $\psi_K(x_L)=n_L^K(\psi_L(x_L))$ and $\psi_K(t_L)=n_L^K(\psi_L(t_L))$.  It follows from the argument in Case 2 that these assignments produce a well-defined Weyl-equivariant map
\[\psi_K\colon B_\bullet(K)\to A_\bullet(K).\]
We can explicitly see that $\psi_K$ induces an isomorphism on $H_0$ as follows.
For each $p\mid \abs{K}$, let $K_p$ be the unique maximal $p$-cyclic subgroup of $K$. On homology, the element in $B_\bullet(K)$ given by the tensor product of all the elements $x_{K_p}$ represents a primitive $\abs{K}$th root of unity, and thus by the injectivity of field maps (or by analyzing the inductive construction), its image under $\psi_K$ is a primitive $\abs{K}$th root of unity in $H_0(A_\bullet(K))$.

Suppose $K'$ is a proper subgroup of $K$.  We must check the commutativity of the diagram
\[ \xymatrix{ \displaystyle \bigotimes_{\substack{\text{$p$ prime}\\p\mid \abs{K}}} \bigotimes_{\substack{L\leq K,\, L\neq e\\\text{$L$ cyclic $p$-group}}} \QQ[x_L]\otimes E(t_L)\ar[r]^-{\psi_K}& A_\bullet(K)\\
\displaystyle \bigotimes_{\substack{\text{$p$ prime}\\p\mid \abs{K'}}} \bigotimes_{\substack{L'\leq K',\, L'\neq e\\\text{$L'$ cyclic $p$-group}}} \QQ[x_{L'}]\otimes E(t_{L'})\ar[r]^-{\psi_{K'}}\ar[u]^-{n_{K'}^K} &A_\bullet(K')\ar[u]_-{n_{K'}^K}
}\]
Consider $x_{L'}\in B_\bullet(K')$. (The case for $t_{L'}$ is exactly the same.) We then compute:
\begin{align*}
\psi_K(n_{K'}^K(x_{L'}))& =\psi_K(x_{L'})\\
&=n_{L'}^K(\psi_{L'}(x_{L'}))\\
&=n_{K'}^K(n_{L'}^{K'}(\psi_{L'}(x_{L'})))\\
&=n_{K'}^K(\psi_{K'}(n_{L'}^{K'}(x_{L'})))\\
&=n_{K'}^K(\psi_{K'}(x_{L'})).
\end{align*}
The first two equalities follow from the definition of the norm shadows in $B_\bullet$ and the construction of the map $\psi_K$.  The third equality is the usual functoriality of norm shadows.  The fourth equality follows from the inductive hypothesis that $\psi_{K'}$ commute with norm shadows for the proper subgroup $K'<K$.  The final equality again follows from the definition of the norm shadows in $B_\bullet$.\medskip

\noindent \textbf{Case 4.} \emph{$K$ is non-cyclic.}  In this case
\[B_\bullet(K)=M_\bullet \otimes \bigotimes_{\substack{\text{$p$ prime}\\p\mid \abs{K}}}\bigotimes_{\substack{L\leq K,\, L\neq e\\\text{$L$ cyclic $p$-group}}} \QQ[x_L]\otimes E(t_L).\]
This is the coproduct in CDGAs of $M_\bullet$ and the CDGA
\[\bigotimes_{\substack{\text{$p$ prime}\\p\mid \abs{K}}}\bigotimes_{\substack{L\leq K,\, L\neq e\\\text{$L$ cyclic $p$-group}}} \QQ[x_L]\otimes E(t_L)\]
so it suffices to define $\psi_K$ on each of these tensor factors.  On the second factor, we may proceed exactly as in Case 3: for a $p$-cyclic subgroup $L\leq K$, we define $\psi_K(x_L)=n_L^K(\psi_L(x_L))$ and $\psi_K(t_L)=n_L^K(\psi_L(t_L))$.  To define $\psi_K$ on $M_\bullet$, recall that by hypothesis $H_0(A_\bullet(K))=0$. Observe that in the category of unital CDGAs, this implies that the $H_*(A_\bullet(K))=0$ and that the image of the unit map $\QQ\to A_\bullet$ vanishes on homology.  That is, either $A_\bullet(K)=0$ or else the multiplicative identity  $1\in A_0(K)$ must be in the image of the differential on $A_\bullet(K)$.  Thus we may choose a Weyl-fixed element $\overline{a}\in A_1(K)$ such that $d(\overline{a})=1$ and define $\psi_K(a)=\overline{a}$. 

By the same argument as in Case 3, this definition of $\psi_K$ is functorial on the shadows of the norm maps.  Since the homology on both sides is zero, the map $\psi_K$ is an isomorphism on homology.
\end{proof}

\section{Uniqueness results for $\KU_G$ and $\ku_G$}\label{sec:5}
In the previous section we constructed a resolution for $H_*(\Theta(H\rep_G))$ with the property that it admits a map to any $\CDGA$ whose degree zero homology is isomorphic to $H_0(\Theta(H\rep_G))$, and that this map induces an isomorphism on homology in degree zero. In this section we use this construction to prove our main results.

\begin{thm} \label{thm:addingbetas}
Let $G$  be a finite abelian group. For any $A_\bullet \in \A(E_\infty^G(G))$ with $H_*(A_\bullet)$ objectwise isomorphic to $H_*(\theta(\KU_G))$, there exists a zig-zag of weak equivalences in $\A(E_\infty^G(G))$ from $A_\bullet$ to $\Theta(\KU_G)$.
\end{thm}
\begin{proof}
For any such $A_\bullet$, observe $H_0(A_\bullet)\cong H_0(\theta(H\rep_G))$ by Theorem \ref{mainresultsprevpaper}. Let $B_\bullet$ be the resolution for $H_*(\Theta(H\rep_G))$ given in Construction \ref{constructH0resolution}. As shown in \cref{lem:goodres}, we have a map $\psi\colon B_\bullet \to A_\bullet$ that induces an isomorphism of diagrams on degree zero homology. Our strategy will be to construct a diagram of CDGAs $D_\bullet$ using $B_\bullet$ and then extend the map $\psi$ to give a weak equivalence $\Psi\colon D_\bullet \to A_\bullet$.

Define a complex $D_\bullet(K)$ for each $K\leq G$ by
	\[
	D_\bullet(K)= B_\bullet(K) \otimes \QQ[\gamma_K, \bar{\gamma}_K] \otimes E(y_K) \quad \text{where}
	\]
	\[
	 |\gamma_K|=2,\quad |\bar{\gamma}_K|=-2, \quad |y_K|=1, \quad d(y_K)=\gamma_K\bar{\gamma}_K-1,
	 \]
and the Weyl group action on $\gamma_K$, $\bar{\gamma}_K$, and $y_K$ is trivial. From the K\"unneth theorem, identifying $[\gamma_K]$ with $\beta_K$, we observe
	\[
	H_*(D_\bullet(K))\cong H_*(B_\bullet(K))\otimes \QQ[\beta_K, \beta_K^{-1}]\cong H_*(\Theta(\KU_G)(K)),
	\]
so the complex has the correct objectwise homology. To define norm shadows, suppose $L\leq K\leq G$ and extend the norm shadows in $B_\bullet$ by making the following assignments on the additional generators:
	\[
	n_L^K(\gamma_L)=\gamma_K, \quad n_L^K(\bar{\gamma}_L)=\bar{\gamma}_K, \quad \text{and} \quad n_L^K(y_L)=y_K.
	\]
Observe that on homology, the norm shadows in $D_\bullet$ satisfy $[\gamma_L] \mapsto [\gamma_K]$ for $L\leq K$.

We next extend $\psi$ to a map $\Psi\colon D_\bullet \to A_\bullet$. To do this, first consider the complex at the trivial subgroup, and choose representatives $\alpha_e, \bar{\alpha}_e\in A_\bullet(e)$ such that $[\alpha_e]=\beta_e$ and $[\bar{\alpha}_e]=\beta_e^{-1}$ in homology. Since $[\alpha_e][\bar{\alpha}_e] = \beta_e \beta_e^{-1}=1$, there must exist a class $z_e\in A_1(e)$ such that $d(z_e)=\alpha_e\bar{\alpha}_e-1$. Note we can choose these elements to be Weyl-fixed by \cref{lem:averagingtrick}. Now let $K$ be a subgroup of $G$. Consider the norm shadow
	\[
	n_e^K\colon A_\bullet(e) \to A_\bullet(K).
	\]
Let $\alpha_K=n_e^K(\alpha_e)$, $\bar{\alpha}_K=n_e^K(\bar{\alpha}_e)$, and $z_K=n_e^K(z_e)$. The norm shadow is a map of $\cdga$'s, so we have the relation
	\[
	d(z_K)= d(n_e^K(z_e))=n_e^K(d(z_e))=n_e^K(\alpha_e\bar{\alpha}_e-1)=\alpha_K\bar{\alpha}_K-1.
	\]
Thus $z_K$ is a witness to the fact that $[\alpha_K][\bar{\alpha}_K]=1$ in homology.

We may  thus extend the map $\psi$ to a map $\Psi$ by defining $\gamma_K\mapsto \alpha_K$, $\bar{\gamma}_K\mapsto \bar{\alpha}_K$, and $y_K\mapsto z_K$. By construction, this is a map of $\CDGA$'s at each level and commutes with the shadows of the norm maps.

It remains to check that the map $\Psi$ induces objectwise isomorphisms on homology. Recall from \cref{mainresultsprevpaper} that the homology at each subgroup is either zero or isomorphic to $\QQ(\zeta_n)[\beta^{\pm 1}]$ for some $n$. The only graded $\QQ$-algebra endomorphisms of $\QQ(\zeta_n)[\beta^{\pm 1}]$ are injective because it is a graded field, and furthermore since $\QQ(\zeta_n)[\beta^{\pm 1}]$ is a finite dimensional $\QQ$-vector space in each grading, injectivity implies surjectivity. Thus the map $\Psi$ does indeed induce objectwise isomorphisms on homology.

We have constructed a weak equivalence $\Psi\colon D_\bullet \to A_\bullet$ for any $A_\bullet$ such that $H_*(A_\bullet)$ is objectwise isomorphic to $H_*(\Theta(\KU_G))$. In particular, there exists a zig-zag of weak equivalences in $\A(E_\infty^G(G))$
\[
\begin{tikzcd}
\Theta(\KU_G) \& D_\bullet \arrow{r}{\simeq} \arrow{l}[swap]{\simeq} \& A_\bullet,
\end{tikzcd}
\]
which completes the proof.
\end{proof}

\begin{thm}\label{genuineuniquenessKU} For any finite abelian group $G$, the spectrum $KU_G$ admits a unique structure as a rational genuine-commutative ring $G$-spectrum. That is, if $X$ is a genuine-commutative ring spectrum whose graded Green functor of homotopy groups is isomorphic to that of $KU_G$, then there is a weak equivalence of rational genuine-commutative ring $G$-spectra between $X$ and $KU_G$.
\end{thm}
\begin{proof} Consider the corresponding algebraic object $\Theta(X)\in \A(E_\infty^G(G))$. From Theorem \ref{thm:addingbetas} we have that $\Theta(X)$ is weakly equivalent to $\Theta(\KU_G)$ in $\A(E_\infty^G(G))$. The result then follows from \cref{thm:genuineAlgMod}.
\end{proof}

Such a uniqueness result does not hold for the connective cover $\Theta(\ku_G)$, although it is true that the diagram $\Theta(\ku_G)$ is formal in $\A(E^G_\infty(G))$. The main difference is the polynomial class $\beta_e$ is not invertible and so the norm shadows are not determined by the objectwise homology.  Compare the results of \cref{formalitylittleku} and \cref{non_uniqueness_ku}.

\begin{thm}\label{formalitylittleku} For $G$ finite abelian, the diagram $\Theta(\ku_G)$ is formal in the category $\A(E^G_\infty(G))$.  In fact, for any diagram $A_\bullet$ such that $H_*(A_\bullet)$ is isomorphic as diagrams to $H_*(\Theta(\ku_G))$, there is a zig-zag of weak equivalences relating $A_\bullet$ and $\Theta(\ku_G)$.
\end{thm}
\begin{proof}
Since $\ku_G$ is the connective cover of $\KU_G$, the norm shadows on $\Theta(\ku_G)$ are the truncations of the norm shadows on $\Theta(\ku_G)$.  Writing $H_*(\Theta_K(\ku_G))=V_K[\beta]$, as in \cref{mainresultsprevpaper},  we see that each norm shadow takes $\beta$ to $\beta$ because \cref{thm:addingbetas} implies that this is true in $H_*(\Theta(\KU_G))$.

As in the proof of \cref{thm:addingbetas}, we extend the resolution $B_\bullet$ of $\Theta(H\rep_G)$ to one of $\Theta(\ku_G)$.  Concretely, for $K\leq G$, we define
\[ D'_\bullet(K)=B_\bullet(K)\otimes \QQ[\gamma_K]\]
where $\abs{\gamma_K}=2$, $\gamma_K$ is Weyl-fixed, and $d(\gamma_K)=0$.  As in the definition of $D_\bullet$, the norm shadows on $D'_\bullet$ are given by the norm shadows on $B_\bullet$ and the assignment $\gamma_L\mapsto \gamma_K$ for $L<K$. Thus by definition, $H_*(D'_\bullet(K))=V_K[\beta]$, with norm shadows taking $\beta$ to $\beta$.

By a similar argument to that in the proof of \cref{thm:addingbetas}, we can construct a map of diagrams $\psi\colon D'_\bullet\to A_\bullet$ inducing a weak equivalence on homology whenever the homology of $A_\bullet$ is isomorphic to that of $D'_\bullet$ via an isomorphism of diagrams.
\end{proof}

This shows that the diagram $\Theta(\ku_G)$ is determined by the diagram $H_*(\Theta(\ku_G))$.  However, the objectwise description of $H_*(\theta(\ku_G))$ given in \cref{mainresultsprevpaper} is not sufficient to determine $\Theta(\ku_G)$.

\begin{lemma}\label{non_uniqueness_ku}
Suppose $G$ is a nontrivial finite abelian group. Then there exists $A_\bullet \in  \A(E_\infty^G(G))$ such that $H_*(A_\bullet)$ is objectwise isomorphic to $H_*(\theta(\ku_G))$ but not isomorphic to $H_*(\Theta(\ku_G))$ as a diagram in $ \A(E_\infty^G(G))$.
\end{lemma}
\begin{proof}
We prove this by constructing such an $A_\bullet$. For subgroups $K\leq G$ define $A_\bullet(K)$ to be the tensor product $B_\bullet(K)\otimes \QQ[\beta_K]$ where $\abs{\beta_K}=2$, $d(\beta_K)=0$, and the Weyl group action on $\beta_K$ is trivial, and where $B_\bullet$ is the diagram defined in Construction \ref{constructH0resolution}. Extend the norm shadows from $B_\bullet$ to $A_\bullet$ via $n_L^K(\beta_L)=0$ for all subgroups $L< K\leq G$. By the K\"unneth Theorem, Lemma \ref{lem:goodreshomology}, and Theorem \ref{mainresultsprevpaper}, we see that $H_*(A_\bullet)$ is objectwise isomorphic to $H_*(\theta(\ku_G))$, with the induced norm shadows on $H_*(A_\bullet)$ all satisfying $\beta_K\mapsto 0$.

On the other hand, the shadows of the norm maps in $H_*(\Theta(\ku_G))$ are restrictions of the shadows of the norms in the diagram of fields $H_*(\Theta(\KU_G))$ because $\ku_G$ is the connective cover of $\KU_G$. Thus they must be injective maps at each level. We conclude $H_*(A_\bullet)$ is not isomorphic to $H_*(\Theta(\ku_G))$ in $\A(E_\infty^G(G))$.
\end{proof}

\begin{rem}
We can in fact construct multiple non-isomorphic objects in $\A(E_\infty^G(G))$ whose objectwise homology is the same as $H_*(\theta(\ku_G))$. The number of non-equivalent genuine-commutative ring structures on $H_*(\theta(\ku_G))$ depends on the cyclic subgroup lattice of $G$. For example, if $G=C_{p^2}$ we can define the shadows of the norms in exactly four ways (up to isomorphism).  We depict these choices as diagrams of the form indicated below on the far left.
\begin{align*}
\xymatrix{ A_\bullet(C_{p^2}) & \QQ(\zeta_{p^2})[\beta] & \QQ(\zeta_{p^2})[\beta] & \QQ(\zeta_{p^2})[\beta] & \QQ(\zeta_{p^2})[\beta]  \\
A_\bullet(C_p)\ar[u]^{n_{C_p}^{C_{p^2}}} &\QQ(\zeta_p)[\beta]\ar[u]^{\beta\mapsto 0} & \QQ(\zeta_p)[\beta]\ar[u]^{\beta \mapsto\beta} & \QQ(\zeta_p)[\beta]\ar[u]^{\beta\mapsto 0} & \QQ(\zeta_p)[\beta]\ar[u]^{\beta\mapsto \beta}\\
A_\bullet(e)\ar[u]^{n_e^{C_p}} &
\QQ[\beta]\ar[u]^{\beta\mapsto 0} & \QQ[\beta]\ar[u]^{\beta\mapsto 0} & \QQ[\beta]\ar[u]^{\beta\mapsto\beta} & \QQ[\beta]\ar[u]^{\beta\mapsto\beta}
}
\end{align*}
 In each diagram, in degree zero the norm shadows are given by the standard inclusions of fields, and in higher degrees determined by the indicated image of $\beta$.
In general there are $2^{n}$ non-isomorphic extensions for $C_{p^n}$ given by the two choices, zero versus nonzero, of the image of the Bott element $\beta$ at each subgroup.

The number of genuine-commutative ring structures on $H_*(\theta(\ku_G))$ becomes even more interesting when considering other cyclic groups. For example when $G=C_{pq}$ for $p$ and $q$ distinct primes, there are four norm shadows and thus at most $2^4 =16$ choices.  Some of these do not give rise to a commutative diagram, such as the combination
\[
n_{e}^{C_{p}}(\beta_{e})=\beta_{C_p},\quad \quad  n_{C_p}^{C_{pq}}(\beta_{C_p})=\beta_{C_{pq}}, \quad \quad n_{e}^{C_{q}}(\beta_{e})=0, \quad \quad n_{C_q}^{C_{pq}}(\beta_{C_q})=0.
\]
After removing the options where $n_{C_p}^{C_{pq}}n_{e}^{C_{p}}(\beta_e)\neq n_{C_q}^{C_{pq}}n_{e}^{C_{q}}(\beta_e)$, one sees there are exactly $10$ genuine-commutative ring structures up to isomorphism in $\A(E_\infty^G(G))$  when $G=C_{pq}$.
\end{rem}

The result of \cref{non_uniqueness_ku} implies a non-uniqueness result for $\ku_G$ that stands in contrast to the result of \cref{genuineuniquenessKU} for periodic K-theory.

\begin{thm}\label{thm:non_uniqueness_ku}
Let $G$ be a non-trivial finite abelian group. There exists a rational genuine-commutative ring G-spectrum $X$ whose underlying Green functor of homotopy groups is isomorphic to that of $\ku_G$ but which is not weakly equivalent to $\ku_G$ as genuine-commutative rational $G$-spectrum. That is, $X$ is weakly equivalent to $\ku_G$ in the category of rational naive-commutative ring $G$-spectra but not in the category of rational genuine-commutative ring $G$-spectra.
\end{thm}
\begin{proof}
Consider the diagram $A_\bullet$ from Lemma \ref{non_uniqueness_ku}. By \cref{thm:genuineAlgMod} there is a corresponding genuine-commutative ring $G$-spectrum $X$ such that $\Theta(X)$ and $A_\bullet$ are weakly equivalent. By the uniqueness of the naive-commutative structure on $\ku_G$ from \cite[Corollary 5.10]{BHIKM}, there exists a zig-zag of weak equivalences between $X$ and $\ku_G$ in the category of naive-commutative rational ring spectra. But since by construction $H_*(\Theta(X))$ is not isomorphic to $H_*(\Theta(\ku_G))$, the diagram $\Theta(X)$ cannot be weakly equivalent to $\Theta(\ku_G)$, and thus $X$ cannot be weakly equivalent to $\ku_G$ as genuine-commutative rational ring $G$-spectra.
\end{proof}

We end this section by clarifying the conclusions that can be drawn from our uniqueness results. As noted in Section \ref{sect:reviewalgmodels}, there is a forgetful functor from rational genuine-commutative ring $G$-spectra to rational naive-commutative ring $G$-spectra \[
U\colon E_\infty^G\text{-alg}(G\Sp_\bQ) \to E_\infty^1\text{-alg}(G\Sp_\bQ).
\]
The main result in Theorem \ref{genuineuniquenessKU} implies that if $X\in E_\infty^G\text{-alg}(G\Sp_\bQ)$ satisfies $U(X)\simeq U(\KU_G)$, then $X\simeq \KU_G$ in $E_\infty^G\text{-alg}(G\Sp_\bQ)$. Note, however, that these results do not imply that any naive-commutative ring $G$-spectrum $Y$ that is weakly equivalent to $\KU_G$ must admit a genuine-commutative ring structure. In other words, not every object that is weakly equivalent to $\KU_G$ in $E_\infty^1\text{-alg}(G\Sp_\bQ)$ arises from an object in $E_\infty^G\text{-alg}(G\Sp_\bQ)$ by forgetting the norm maps. Below we construct a naive-commutative ring $G$-spectrum that is weakly equivalent to $\KU_G$ but does not admit a genuine-commutative ring structure.

\begin{ex}\label{e:no_extension_to_genuine}
Let $G=C_{p^2}$ where $p\neq 2$.  We construct an $A_\bullet\in\A(E^1_\infty(G))$ that is not in the image of the forgetful functor from $\A(E^G_\infty(G))$ to $\A(E^1_\infty(G))$.  Via \cref{lem:diagramOfAlgMod}, there thus exists a naive-commutative ring spectrum $X$ corresponding to $A_\bullet$ that does not admit a genuine-commutative ring structure. The object $A_\bullet$ is defined as follows:
\begin{align*}
	A_\bullet(C_{p^2})& =\QQ[x_{p^2}]\otimes E(y_{p^2}) &
	A_\bullet({C_{p}})& =\QQ (\zeta_p ) &
	A_\bullet({e})&=\QQ
	\end{align*}
where in the first CDGA $\abs{x_{p^2}}=0$, $\abs{y_{p^2}}=1$,  and $d(y_{p^2})=\Phi_{p^2}(x_{p^2})$ and the latter two CDGAs are concentrated in degree zero with zero differential. Note that the homology of  $A_\bullet({C_{p^2}})$ is $\QQ(\zeta_{p^2})$ concentrated in degree zero.  Hence the homology of $A_\bullet$ is isomorphic to the homology of $\theta(H\rep_G)$.  By \cite[Lemma 5.5]{BHIKM}, the complex $\theta(H\rep_G)$ is formal; hence there is a zig-zag of weak equivalences of rational naive-commutative ring  $G$-spectra between $X$ and  $H\rep_G$.  We show that $A_\bullet$ is not the underlying naive-commutative object of a well-defined object in $\A(E_\infty^G(G))$.

Suppose that a norm shadow
\[n\colon A_\bullet({C_{p}}) \to A_\bullet({C_{p^2}})\] exists.
By a degree argument we must have that $n(\zeta_p ) = f(x_{p^2}),$ where  $f(x_{p^2})$ is some polynomial in $\QQ[x_{p^2}]$, and therefore
\[\left(n(\zeta_p )\right)^p=\left(f(x_{p^2})\right)^p.\] Furthermore, since the shadows of the norm maps are maps of commutative rings, we also have the identity
\[\left(n(\zeta_p )\right)^p=n((\zeta_p) ^p)=n(1)=1.\]
As a consequence, we get $(f(x_{p^2}))^p=1$ which only holds when $f(x_{p^2})=1$. However, in degree zero the map $n$ is a map from a field to a ring and thus must be injective.  Since $1\in A_\bullet(C_{p^2})$ is also the image of the unit $1\in A_\bullet(C_p)$, this leads to a contradiction.

Since $\KU_G$ and $H\rep_G$  agree at the zeroth degree, this example can be extended to an example of a naive-commutative ring spectrum $X$ that is isomorphic to $\KU_G$ in the homotopy category of naive-commutative ring $G$-spectra but does not admit a genuine-commutative ring structure.
\end{ex}

\begin{rem} At $p=2$, one has $\QQ(\zeta_2)=\QQ$, and so the complex in $\A(E^1_\infty(C_4))$ constructed in \cref{e:no_extension_to_genuine} does underlie a complex in $\A(E^{C_4}_\infty(C_4))$ in which all norm shadows are the unit map.  In this case, one can instead construct slightly more complicated examples of complexes in $\A(E^1_\infty(C_4))$ that do not arise from an object in $\A(E^{C_4}_\infty(C_4))$.
\end{rem}

\bibliography{wit3references}
\bibliographystyle{plain}

\end{document}